\documentclass[a4paper]{amsart}
\usepackage[english]{babel}
\usepackage{epsfig}
\usepackage{amsmath}
\usepackage{amsfonts}
\usepackage{amssymb}
\usepackage{amsthm}
\usepackage{epic}
\usepackage{fancyhdr}
\usepackage{fancybox}

\usepackage{amsopn}
\usepackage{latexsym}

\usepackage{hyperref}
\hypersetup{   pdfborder={0 0 0},   
              colorlinks = false}

\newcommand{\R}{\mathbb R}
\newcommand{\Z}{\mathbb Z}
\newcommand{\Q}{\mathbb Q}
\newcommand{\N}{\mathbb N}

\newtheorem{Theorem}[equation]{Theorem}
\newtheorem{Proposition}[equation]{Proposition}
\newtheorem{Lemma}[equation]{Lemma}

\theoremstyle{definition}
\newtheorem{Definition}[equation]{Definition}
\theoremstyle{remark}
\newtheorem{Rmk}[equation]{Remark}

\title[Diophantine approximations with prescribed quality]{Finding simultaneous Diophantine approximations with prescribed quality}

\author[Wieb {\sc Bosma}]{{\sc Wieb} BOSMA}
\address{Wieb {\sc Bosma}\\
Mathematics\\
Radboud University Nijmegen\\
PO Box 9010, 6500 GL \mbox{Nijmegen}\\
 the Netherlands}
\email{bosma@math.ru.nl }

\author[Ionica {\sc Smeets}]{{\sc Ionica} SMEETS}
\address{Ionica {\sc Smeets}\\
Mathematical Institute\\
Leiden University\\
Niels Bohrweg 1, 2333 CA Leiden\\
 the Netherlands}
\email{ionica.smeets@gmail.com}

\begin{document}
\maketitle

\begin{abstract}
We give an algorithm that finds a sequence of approximations with Dirichlet coefficients bounded by a constant only depending on the dimension. The algorithm uses the LLL-algorithm for lattice basis reduction. We present a version of the algorithm that runs in polynomial time of the input.
\end{abstract}

\section{Introduction}
The regular continued fraction algorithm is a classical algorithm to find good approximations by rationals for a given number $a \in \R\setminus \Q$. Dirichlet~\cite{Di}  proved that for every $a \in \R\setminus \Q$ there are infinitely many integers $q$ such that
\begin{equation}
\label{eq: dirichlet m=n=1}
\|q \,a \| < q^{-1},
\end{equation}
where $\|x\|$ denotes the distance between $x$ and the nearest integer. The exponent $-1$ of $q$ is minimal; if it is replaced by any number $x<-1$, then there exist $a \in \R \setminus \Q$ such that only finitely many integers $q$ satisfy $||q\,a|| < q^{x}$.

Hurwitz \cite{H} proved that the continued fraction algorithm finds, for every $a \in \R\setminus \Q$, an infinite sequence of increasing integers $q_n$ with
$$
\|q_n \, a  \| < \frac{1}{\sqrt{5}} \, q_n^{-1}.
$$
This result is sharp:  if the constant $ \frac{1}{\sqrt{5}} $ is replaced by any smaller one, then the statement becomes incorrect. Legendre~\cite{L} showed that the continued fraction algorithm finds all good approximations, in the sense that if 
$$
\|q  \,a \| <\frac12 \, q^{-1},
$$
then $q$ is one of the $q_n$ found by the algorithm.

As to the generalization of approximations in higher dimensions Dirichlet proved the following theorem;  See Chapter II  of~\cite{Schmidt}.
\begin{Theorem}
\label{th: ILLL diri}
Let an $n \times m$ matrix $A$ with entries $a_{ij} \in \R \setminus \Q$ be given and  suppose that $1, a_{i1} ,\dots, a_{im}$ are linearly independent over $\Q$ for some $i$ with 
$1 \leq i \leq n$. There exist infinitely many coprime $m$-tuples of integers $q_1,\dots,q_m$ such that with\\ \mbox{$q = \displaystyle{\max_j |q_j|}\geq 1$}, we have
\begin{equation}
\max_i \| q_1 \,a_{i1} + \dots + q_m \,a_{im} \|  < q^\frac{-m}{n}. 
\end{equation}
The exponent $\frac{-m}{n}$ of $q$ is minimal.
\end{Theorem}

\begin{Definition}
\label{def: dirichlet quality}
Let an $n \times m$ matrix $A$ with entries $a_{ij} \in \R \setminus \Q$  be given. The \emph{Dirichlet coefficient} of an $m$-tuple $q_1,\dots,q_m$ is defined as 
$$ 
q^\frac{m}{n} \displaystyle{ \max_i \| q_1 \,a_{i1} + \dots + q_m \,a_{im} \| }\,.
$$ 
 \end{Definition}


%


The proof of the theorem does not give an efficient way of finding a series of approximations with a Dirichlet coefficient of 1. For the case $m=1$ the first multi-dimensional continued fraction algorithm was given by Jacobi~\cite{Jacobi}. Many more followed, see for instance Perron~\cite{Perron}, Brun~\cite{Brun1}, Lagarias~\cite{Lag} and Just~\cite{Just}.  Brentjes~\cite{Brentjes} gives a detailed history and description of such algorithms. Schweiger's book~\cite{Schweiger} gives a broad overview. For $n=1$ there is amongst others the algorithm by Ferguson and Forcade~\cite{FF}.  However, there is no efficient algorithm that guarantees to find a series of approximations with Dirichlet coefficient smaller than 1. In 1982 the LLL-algorithm for lattice basis reduction was published in~\cite{LLL}. The authors noted that their algorithm could be used for finding Diophantine approximations of given rationals with Dirichlet coefficient only depending on the dimension; see~(\ref{eq: LLL quality}). 

Just~\cite{Just} developed an algorithm based on lattice reduction that detects $\Z$-linear dependence in the $a_i$, in this case $m=1$. If no such dependence is found her algorithm returns integers $q$ with 
\[\max_i \|qa_i \| \leq  c \left(\sum_{i=1}^n a_i^2\right)^{1/2} q^{-1/(2n(n-1))},\]
where $c$ is a constant depending on $n$. The exponent $-1/(2n(n-1))$ is larger than the Dirichlet exponent $-1/n$. 

Lagarias~\cite{Lag2}  used the LLL-algorithm in a series of lattices to find good approximations for the case $m=1$. Let $a_1,\dots,a_n \in \Q$ and let there be a \mbox{$Q \in \N$} with \mbox{$1 \leq Q \leq N$} such that $\displaystyle{\max_j || Q \, a_j|| <\varepsilon}$. Then Lagarias' algorithm on input $a_1,\dots,a_n$ and $N$ finds in polynomial time a $q$ with $1 \leq q \leq 2^\frac{n}{2}N$ such that $\displaystyle{\max_j ||q \, a_j|| \leq \sqrt{5n}2^\frac{d-1}{2}\varepsilon}$. The main difference with our work is that Lagarias focuses on the quality $ ||q \, a_j|| $, while we focus on Dirichlet coefficient $ q^\frac{1}{n} ||q \, a_j||$. Besides that we also consider the case $m>1$. 

The main result of the present paper is an algorithm that by iterating the LLL-algorithm gives a series of approximations of given rationals with optimal Dirichlet exponent. Where the LLL-algorithm gives one approximation our dynamic algorithm gives a series of successive approximations. To be more precise: For a given $n \times m$-matrix $A$ with entries $a_{ij} \in \Q$ and a given upper bound $q_{\max}$  the algorithm returns a sequence of $m$-tuples $q_1,\dots,q_m$ such that for every $Q$ with $2^\frac{(m+n+3)(m+n)}{4m}\leq Q \leq q_{\max}$ one of these $m$-tuples satisfies 
\begin{eqnarray*}
&&\max_j |q_j| \leq Q \textrm{ and }\\
&&\max_i \|q_1 a_{i1} + \dots +  q_m a_{im}   \| \leq 2^{\frac{(m+n+3)(m+n)}{4n}} Q^\frac{-m}{n}.
\end{eqnarray*}

The exponent $-m/n$ of $Q$ can not be improved and therefore we say that these approximations have optimal Dirichlet exponent.

Our algorithm is a multi-dimensional continued fraction algorithm in the sense that we work in a lattice basis and that we
only interchange basis vectors and add integer multiples of basis vectors to  another. Our algorithm differs from other multi-dimensional continued fraction algorithms in that the lattice is not fixed across the iterations. 

In Lemma~\ref{lem: what we find} we show that if there exists an extremely good approximation, our algorithm finds a very good one. We derive in Theorem~\ref{th: what you find} how the output of our algorithm gives a lower bound on the quality of possible approximations with coefficients up to a certain limit. If the lower bound is positive this proves that there do not exist linear dependencies with all coefficients $q_i$ below the limit. 

In Section~\ref{sec: polynomial} we show that a slightly modified version of our algorithm runs in polynomial time. In Section~\ref{sec: numerical} we present some numerical data.

\section{Lattice reduction and the LLL-algorithm}
In this section we give the definitions and results that we need for our algorithm.

Let $r$ be a positive integer. A subset $L$ of the $r$-dimensional real vector space $\R^r$ is called a \emph{lattice} if there exists a basis $b_1,\dots,b_r$ of $\R^r$ such that
$$
L = \sum_{i=1}^r \Z b_i = \left\{ \sum_{i=1}^r z_i b_i; z_i \in \Z \, (1 \leq i \leq r)\right\}.
$$
We say that $b_1,\dots,b_r$ is a \emph{basis} for $L$. The \emph{determinant} of the lattice $L$ is defined by $| \det(b_1,\dots,b_r)|$ and we denote it as $\det(L)$.

For any linearly independent $b_1,\dots,b_r \in \R^r$ the Gram-Schmidt process yields an orthogonal basis $b_1^*,\dots,b_r^*$ for $\R^r$, by defining inductively
\begin{eqnarray}
\label{eq: gram schmidt b} &&b_i^* = b_i - \sum_{j=1}^{i-1} \mu_{i  j} b_j^* \quad \textrm{for } 1 \leq i \leq r \quad \textrm{and}\\
\nonumber &&\mu_{i  j} = \frac{(b_i,b_j^*)}{(b_j^*,b_j^*)},
\end{eqnarray}
where $(\,,)$ denotes the ordinary inner product on $\R^r$.

We call a basis $b_1,\dots,b_r$ for a lattice $L$ \emph{reduced} if
$$
|\mu_{ij}| \leq \frac12 \quad \textrm{for } 1 \leq j < i \leq r \\
$$ 
and 
$$
|b_i^* + \mu_{i i-1}b_{i-1}^*|^2  \leq  \frac{3}{4} |b_{i-1}^*|^2 \quad \textrm{ for } 1 \leq i \leq r,
 $$
where $|x|$ denotes the Euclidean length of $x$.

The following two propositions were proven in~\cite{LLL}.

\begin{Proposition}
\label{prop: LLL b1}
Let $b_1,\dots,b_r$ be a reduced basis for a lattice $L$ in $\R^r$. Then we have
\begin{eqnarray*}
&(i)& |b_1| \leq 2^{(r-1)/4}\, \bigl( \det(L) \bigr)^{1/r}, \quad \\
&(ii)& |b_1|^2 \leq 2^{r-1} \, |x|^2, \quad \textrm{for every } x \in L, x \neq 0,\\
&(iii)& \prod_{i=1}^r |b_i| \leq 2^{r(r-1)/4} \det(L).
\end{eqnarray*}
\end{Proposition}

\begin{Proposition}
\label{prop: LLL complexity}
Let $L \subset \Z^r$ be a lattice with a basis $b_1,b_2,\ldots,b_r$, and let $F \in \R$, $F \geq 2$, be such that $|b_i|^2 \leq F$ for $1 \leq i \leq r$. Then the number of arithmetic operations needed by the LLL-algorithm is $O(r^4 \log F)$ and the integers on which these operations are performed each have binary length $O(r \log F)$.
\end{Proposition}

In the following Lemma the approach suggested in the original LLL-paper for finding (simultaneous) Diophantine approximations is generalized to the case $m >1$.
\begin{Lemma}
\label{lem: LLL setup}
Let an $n \times m$-matrix $A$ with entries $a_{ij}$ in $\R$ and an $\varepsilon \in (0,1)$ be given. Applying the LLL-algorithm to the basis formed by the columns of the $(m+n) \times (m+n)$-matrix
\begin{equation}
\label{eq:lll begin linear system}
B=
\left[\begin{array}{ccccccc}
1 & 0 & \dots & 0 & a_{11} & \dots & a_{1m}\\
0 & 1 & \ddots & 0 & a_{21} & \dots & a_{2m}\\
\vdots &  & & \vdots & \vdots & & \vdots\\
0 & \dots & 0& 1 & a_{n1} & \dots & a_{nm}\\
0 & \dots &0 & 0 &c &  & 0\\
\vdots &  &\vdots & \vdots & & \ddots & \\
0 & \dots &0 & 0 &0 &  & c\\
\end{array}\right],\end{equation}
with $c = \left( 2^{-\frac{m+n-1}{4}} \varepsilon  \right)^\frac{m+n}{m}$ yields an $m$-tuple $q_1,\dots,q_m \in \Q$ with
\begin{eqnarray}
\max_j |q_j| &\leq&2^\frac{(m+n-1)(m+n)}{4m}\varepsilon^\frac{-n}{m}  \label{eq: bound q lll}
 \textrm{ and}\\
\max_i \| q_1 a_{i1} + \dots + q_m a_{im}\| &\leq&\varepsilon  \label{eq: bound quality lll}.
\end{eqnarray}
\end{Lemma}
\begin{proof}

The LLL-algorithm finds a reduced basis $b_1,\dots,b_{m+n}$ for this lattice. Each vector in this basis can be written as
$$ 
\left[\begin{array}{c}
q_1a_{11}+ \dots + q_m a_{1m} - p_1\\
\vdots\\
q_1a_{n1}+ \dots + q_m a_{nm} - p_n\\
cq_1\\
\vdots \\
c q_m
\end{array}
\right],
$$
with $ p_i \in \Z$ for $1\leq i \leq n$  and $ q_j \in \Z$ for $1\leq j \leq m$.

Proposition~\ref{prop: LLL b1}(i) gives an upper bound for the length of the first basis vector,
\[|b_1| \leq 2^\frac{m+n-1}{4}c^\frac{m}{m+n}.\]
From this vector $b_1$ we find integers $q_1,\dots,q_m$, such that
\begin{eqnarray}
\label{eq: q in c}
\max_j |q_j| &\leq&2^\frac{m+n-1}{4}c^\frac{-n}{m+n}  \textrm{ and } \\
\label{eq: quality in c}
\max_i \| q_1 a_{i1} + \dots + q_m a_{im}\| &\leq& 2^\frac{m+n-1}{4}c^\frac{m}{m+n} .
\end{eqnarray}
Substituting $c = \left( 2^{-\frac{m+n-1}{4}} \varepsilon  \right)^{\frac{m+n}{m}}$ gives the results.
\end{proof}

From equations~(\ref{eq: q in c}) and~(\ref{eq: quality in c}) it easily follows that the $m$-tuple $q_1,\dots,q_m$ satisfies
\begin{equation}
\label{eq: LLL quality}
\max_i \|q_1 a_{i1} + \dots + q_m a_{im}\| \leq 2^\frac{(m+n-1)(m+n)}{4n}q^\frac{-m}{n},
\end{equation}
where $q = \displaystyle{\max_j |q_j|}$, so the approximation has a Dirichlet coefficient of at most $2^\frac{(m+n-1)(m+n)}{4n}$. 

\section{The Iterated LLL-algorithm}
\label{sec: algorithm}

We iterate the LLL-algorithm over a series of lattices to find a sequence of approximations. We start with a lattice determined by a basis of the form~(\ref{eq:lll begin linear system}). After the LLL-algorithm finds a reduced basis for this lattice, we decrease the constant $c$ by dividing the last $m$ rows of the matrix by a constant greater than 1. By doing so, $\varepsilon$ is divided by this constant to the power~$\frac{m}{m+n}$. We repeat this process until the upper bound~(\ref{eq: bound q lll}) for $q$ guaranteed by the LLL-algorithm exceeds a given upper bound $q_{\max}.$  Motivated by the independence on $\varepsilon$ of~(\ref{eq: LLL quality}) we ease notation by fixing  $\varepsilon=\frac12$.

Define
\begin{equation}
\label{eq: number of steps}
k^\prime := \left\lceil    - \frac{(m+n-1)(m+n)}{4n} + \frac{m \log_2 q_{\max}}{n }   \right\rceil.
\end{equation}

\shadowbox
{ \begin{minipage}[c]{\textwidth}
\textbf{Iterated LLL-algorithm (ILLL)}\\

\textbf{Input}\\ 
An $n \times m$-matrix $A$ with entries $a_{ij}$ in $\R$.\\
An upper bound $q_{\max}>1$.\\

\textbf{Output}\\
For each integer $k$ with $1 \leq k \leq k^\prime$, see~(\ref{eq: number of steps}), we obtain a vector $q(k) \in \Z^m$ with 
\begin{eqnarray}
\label{eq: q k bound}
\max_j|q_j(k)| &\leq& 2^\frac{(m+n-1)(m+n)}{4m} \, 2^\frac{kn}{m},\\
\label{eq: quality k bound}
\max_i \|q_1(k) \,a_{i1} + \dots +  q_m(k) \, a_{im}   \| & \leq &  \frac{1}{2^k}.
\end{eqnarray}

 \textbf{Description of the algorithm}
\begin{enumerate}
\item Construct the basis matrix $B$ as given in~(\ref{eq:lll begin linear system}) from $A$.
\item Apply the LLL-algorithm to $B$.
\item Deduce $q_1,\dots,q_m$ from the first vector in the reduced basis returned by the LLL-algorithm. 
\item Divide the last $m$ rows of $B$ by $2^\frac{m+n}{m}$
\item Stop if the upper bound for $q$ guaranteed by the algorithm~(\ref{eq: q k bound}) exceeds  $q_{\max}$; else go to step 2. 
\end{enumerate}
\end{minipage}
} \\

\begin{Rmk}The number  $2^\frac{m+n}{m}$ in step 4 may be replaced by $ d^\frac{m+n}{m}$ for any real number $d>1$. When we additionally set $\varepsilon = \frac{1}{d} $ this yields 
 that
\begin{eqnarray}
\label{eq: q(k) in d}
\max_j |q_j(k)| &\leq&2^\frac{(m+n-1)(m+n)}{4m}d^\frac{kn}{m}  \quad \textrm{ and } \\
\label{eq: quality in d}
\max_i \| q_1(k) a_{i1} + \dots + q_m(k) a_{im}\| &<& d^{-k}.
\end{eqnarray}
In the theoretical part of this paper we always take $d = 2$ corresponding to our choice $\varepsilon=\frac12$.
 \end{Rmk}

\begin{Lemma}
\label{lemma: number steps }
Let an $n \times m$-matrix $A$ with entries $a_{ij}$ in $\R$ and an upper bound $q_{\max}>1$ be given. The number of times the ILLL-algorithm applies the LLL-algorithm on this input equals $k^\prime$ from~{\rm(\ref{eq: number of steps})}.
\end{Lemma}
\begin{proof}
One easily derives the number of times we iterate by solving $k$ from the stopping criterion~(\ref{eq: q k bound})
$$
q_{\max} \leq 2^\frac{(m+n-1)(m+n)}{4m}  2^\frac{kn}{m}, 
$$
\end{proof}
We define
\begin{equation*}
c(k) = c(k\!-\!1) / 2^{\frac{m+n}{m}} \textrm{ for } k >1, \quad \textrm{ where } c(1) = c \textrm{ as given in Lemma~\ref{lem: LLL setup}}.
\end{equation*}
In iteration $k$ we are working in the lattice defined by the basis in~(\ref{eq:lll begin linear system}) with $c$ replaced by $c(k)$. 

\begin{Lemma}
\label{lem kth result}
The $k$-th output, $q(k)$, of the ILLL-algorithm satisfies~{\rm(\ref{eq: q k bound})} and~{\rm(\ref{eq: quality k bound})}. 
\end{Lemma}
\begin{proof}

In step $k$ we use $c(k)=   \left( 2^{-\frac{m+n+3}{4}-k+1} \right)^\frac{m+n}{m}$. Substituting $c(k)$ for $c$ in equations~(\ref{eq: q in c}) and~(\ref{eq: quality in c}) yields~(\ref{eq: q k bound}) and~(\ref{eq: quality k bound}), respectively.
\end{proof}
The following theorem gives the main result mentioned in the introduction. The algorithm returns a sequence of approximations with all coefficients smaller than $Q$, optimal Dirichlet exponent and Dirichlet coefficient only depending on  the dimensions $m$ and $n\,$.

\begin{Theorem}
\label{th: for each Q}
Let an $n \times m$-matrix $A$ with entries $a_{ij}$ in $\R$, and $q_{\max} >1$  be given. The ILLL-algorithm finds a sequence of $m$-tuples $q_1,\dots,q_m$ such that for every $Q$ with $2^\frac{(m+n+3)(m+n)}{4m} \leq Q \leq q_{\max}$ one of these $m$-tuples satisfies 
\begin{eqnarray*}
&&\max_j |q_j| \leq Q \textrm{ and }\\
&&\max_i \|q_1 a_{i1} + \dots +  q_m a_{im}   \| \leq 2^{\frac{(m+n+3)(m+n)}{4n}} Q^\frac{-m}{n}.
\end{eqnarray*}
\end{Theorem}
\begin{proof}
Take $k \in \N$ such that
\begin{equation}
\label{eq: Q and k}
2^\frac{(m+n+3)(m+n)}{4m}\cdot 2^\frac{(k-1)n}{m} \leq Q < 2^\frac{(m+n+3)(m+n)}{4m} \cdot 2^\frac{kn}{m} .
\end{equation}

From Lemma~\ref{lem kth result} we know that $q(k)$ satisfies the inequality
\[
\max_j |q_j(k)| \leq 2^\frac{(m+n+3)(m+n)}{4m} \, 2^\frac{(k-1)n}{m}  \leq Q.
\]
From the right hand side of inequality~(\ref{eq: Q and k}) if follows that $ \frac{1}{2^k} <  2^\frac{(m+n+3)(m+n)}{4n} Q^\frac{-m}{n}.$ From Lemma~\ref{lem kth result} and this inequality  we derive that
$$
\max_i\|q_1(k) \,a_{i1} + \dots +  q_m(k) \, a_{im}   \| \leq  \frac{1}{2^k} <  2^{\frac{(m+n+3)(m+n)}{4n}} Q^\frac{-m}{n}.
$$
\end{proof}

Proposition~\ref{prop: LLL b1}(ii) guarantees that if there exists an extremely short vector in the lattice, then the LLL-algorithm finds a rather short lattice vector. We extend this result to the realm of successive approximations. In the next lemma we show that for every very good approximation, the ILLL-algorithm finds a rather good one not too far away from it.

\begin{Lemma}
\label{lem: what we find}
Let an $n \times m$-matrix $A$ with entries $a_{ij}$ in $\R$, a real number $0<\delta<1$ and an integer $s >1$ be given. If there exists an $m$-tuple $s_1,\dots,s_m$ with
\begin{eqnarray}
\label{eq: S bound} &&s = \max_j |s_j|  >2^\frac{(m+n-1)n}{4m} \left(\frac{n\delta^2}{m}\right)^\frac{n}{2(m+n)}\\
\nonumber \textrm{ and} \\
\label{eq: S quality bound} &&\max_i \|s_1 a_{i1} + \dots +  s_m a_{im}   \| \leq \delta s^\frac{-m}{n},
\end{eqnarray}
 then applying the ILLL-algorithm with 
\begin{equation}
q_{\max} \geq 2^\frac{m^2+m(n-1)+4n}{4m} \left(\frac{m}{n\delta^2} \right)^\frac{n}{2(m+n)} \, s \label{eq: which q} 
\end{equation}
yields an $m$-tuple $q_1,\dots,q_m$ with
\begin{eqnarray}
\label{eq: q for s} 
 \max_j |q_j|  &\leq &2^\frac{m^2+m(n-1)+4n}{4m}\left(\frac{m}{n\delta^2}\right)^\frac{n}{2(m+n)}s \\
 \nonumber \textrm{ and } && \\
 \label{eq:quality for s}  \max_i\|q_1 a_{i1} + \dots +  q_m a_{im} \| & \leq &   2^\frac{m+n}{2} \sqrt{n}\delta s^\frac{-m}{n}.
\end{eqnarray}
\end{Lemma}
\begin{proof}
Let $1\leq k \leq k^\prime$ be an integer. Proposition~\ref{prop: LLL b1}(ii) gives that for each $q(k)$ found by the algorithm
\begin{eqnarray*}
&& \sum_{i=1}^n \|q_1(k) a_{i1} + \dots + q_m(k) a_{im}\|^2  + c(k)^2 \sum_{j=1}^m q_j(k)^2 \\
&& \leq 2^{m+n-1} \left( \sum_{i=1}^n \|s_1 a_{11} + \dots + s_m a_{im}\|^2 + c(k)^2 \sum_{j=1}^m s_j^2 \right).
\end{eqnarray*}
From this and~(\ref{eq: S bound})~and~(\ref{eq: S quality bound}) it follows that
\begin{equation}
\label{eq: step1 qk proof}
\max_i  \|q_1(k) a_{i1} + \dots + q_m(k) a_{im}\|^2 \leq 2^{m+n-1} \left( n\delta^2 s^\frac{-2m}{n}+ c(k)^2 m s^2\right).
\end{equation}

Take the smallest positive integer $K$ such that
\begin{equation}
\label{eq: ck bound in proof}
c(K) \leq \sqrt{\frac{n}{m}}\delta s^{-\frac{m+n}{n}}.
\end{equation}

We find for step $K$ from~(\ref{eq: step1 qk proof}) and~(\ref{eq: ck bound in proof})
$$
\max_i  \|q_1(K) a_{i1} + \dots + q_m(K) a_{im}\| \leq 2^\frac{m+n}{2} \sqrt{n}\delta s^\frac{-m}{n},
$$
which gives~(\ref{eq:quality for s}). 

We show that under assumption~(\ref{eq: which q}) the ILLL-algorithm makes at least $K$ steps. We may assume $K>1$, since the ILLL-algorithm always makes at least 1 step. From Lemma~\ref{lemma: number steps } we find that if $q_{\max}$ satisfies
$$
q_{\max} > 2^\frac{Kn}{m} 2^\frac{(m+n-1)(m+n)}{4m},
$$
then the ILLL-algorithm makes at least $K$ steps. Our choice of $K$ implies 
$$
c(K-1) = \frac{c(1)}{2^\frac{(m+n)(K-2)}{m}}=  \frac{2^{-\frac{(m+n+3)(m+n)}{4m}}}{2^\frac{(m+n)(K-2)}{m}} >  \sqrt{\frac{n}{m}}\delta s^{-\frac{m+n}{n}},
$$ 
and we obtain
$$
2^\frac{Kn}{m} < 2^{-\frac{(m+n-5)n}{4m}} \left(\frac{m}{n\delta^2}\right)^\frac{n}{2(m+n)}s. 
$$
From this we find that 
$$
q_{\max} >  2^\frac{m^2+m(n-1)+4n}{4m} \left(\frac{m}{n\delta^2}\right)^\frac{n}{2(m+n)} s 
$$
is a satisfying condition to guarantee that the algorithm makes at least $K$ steps.

Furthermore, either $2^\frac{-(m+n)}{m} \sqrt{\frac{n}{m}}\delta s^{-\frac{m+n}{n}} < c(K)$ or $K=1$. In the former case we find from~(\ref{eq: q in c}) that 
$$
\max_j |q_j(K)| \leq 2^\frac{m+n-1}{4}c(K)^\frac{-n}{m+n} < 2^\frac{m+n-1}{4} 2^\frac{n}{m}\left(\frac{m}{n\delta^2}\right)^\frac{n}{2(m+n)}s.
$$
In the latter case we obtain from~(\ref{eq: q in c}) 
$$
\max_j |q_j(1)| \leq 2^\frac{m+n-1}{4}c(1)^\frac{-n}{m+n} = 2^\frac{m+n-1}{4}2^\frac{(m+n+3)n}{4m}
$$
and, by~(\ref{eq: S bound}), 
$$
 2^\frac{m+n-1}{4}2^\frac{(m+n+3)n}{4m} =  2^\frac{m+n-1}{4}2^\frac{n}{m}2^\frac{(m+n-1)n}{4m}< 2^\frac{m+n-1}{4} 2^\frac{n}{m}\left(\frac{m}{n\delta^2}\right)^\frac{n}{2(m+n)} s. 
$$
We conclude that for all $K\geq1$
$$
\max_j |q_j(K)| \leq   2^\frac{m^2+m(n-1)+4n}{4m}\left(\frac{m}{n\delta^2}\right)^\frac{n}{2(m+n)}s.
$$
\end{proof}


Note that from~(\ref{eq: q for s}) and~(\ref{eq:quality for s}) it follows that  
\begin{equation}
\label{eq: cor lemma what you find}
q^\frac{m}{n} \max_i\|q_1 a_{i1} + \dots +  q_m a_{im} \| \leq 2^\frac{m^2+m(3n-1)+4n+2n^2}{4n} m^\frac{m}{2(m+n)} (n\delta^2)^\frac{n}{2(m+n)},
\end{equation}
where again $q = \displaystyle{\max_j |q_j|}$.


\begin{Theorem}
\label{th: what you find}
Let an $n \times m$-matrix $A$ with entries $a_{ij}$ in $\R$ and $q_{\max} >1$ be given. Assume that $\gamma$ is such that for every $m$-tuple $q_1,\dots,q_m$ returned by the ILLL-algorithm 
\begin{equation}
\label{eq: cond q}
q^\frac{m}{n} \max_i \| q_1 a_{i1}+\dots q_m a_{im}\| >  \gamma   \textrm{, where } q=\max_j |q_j|. \\ 
\end{equation}
Then every $m$-tuple $s_1,\dots,s_m$ with $ s =   \max_j |s_j| $ and
$$ 2^\frac{(m+n-1)n}{4m} \left(\frac{n\delta^2}{m}\right)^\frac{n}{2(m+n)}< s <2^{-\frac{m^2+m(n-1)+4n}{4m}}\left( \frac{n \delta^2}{m}\right)^\frac{n}{2(m+n)}q_{\max}$$ 
satisfies
$$
s^\frac{m}{n} \max_i \|s_1 a_{i1} + \dots +  s_m a_{im}   \| > \delta,
$$
with
\begin{equation}
\label{eq: set delta}
\delta =2^\frac{-(m+n)(m^2+m(3n-1)+4n+2n^2)}{4n^2}  m^\frac{-m}{2n} n^\frac{-1}{2} \gamma^\frac{m+n}{n} .
\end{equation}
\end{Theorem}
\begin{proof}
Assume that every vector returned by our algorithm satisfies~(\ref{eq: cond q}) and that there exists an  $m$-tuple $s_1,\dots,s_m$ with $s= \max_j |s_j| $ such that
\begin{eqnarray*}
&& 2^\frac{(m+n-1)n}{4m} \left(\frac{n\delta^2}{m}\right)^\frac{n}{2(m+n)} < s <2^{-\frac{m^2+m(n-1)+4n}{4m}}\left( \frac{n \delta^2}{m}\right)^\frac{n}{2(m+n)}q_{\max}\\
&\textrm{and }& s^\frac{m}{n} \max_i \|s_1 a_{i1} + \dots +  s_m a_{im}   \| \leq \delta.
\end{eqnarray*}
From the upper bound on $s$ it follows that $q_{max}$ satisfies~(\ref{eq: which q}). We apply Lemma~\ref{lem: what we find} and find that the algorithm finds an
$m$-tuple $q_1,\dots,q_m$ that satisfies~(\ref{eq: cor lemma what you find}). Substituting $\delta$ as given in~(\ref{eq: set delta}) gives
$$
q^\frac{m}{n} \max_i  \|q_1 a_{i1} + \dots + q_m a_{im}\| \leq \gamma,
$$ 
which is a contradiction with our assumption. 
\end{proof}






\section{A polynomial time version of the ILLL-algorithm}
\label{sec: polynomial}
We have used real numbers in our theoretical results, but in a practical implementation of the algorithm we only use rational numbers. Without loss of generality we may assume that these numbers are in the interval $[0,1]$. In this section we describe the necessary changes to the algorithm and we show that this modified version of the algorithm runs in polynomial time. 

As input for the rational algorithm we take
\begin{itemize}
\item the dimensions $m$ and $n$,
\item  a rational number $\varepsilon \in (0,1)$,
\item  an integer $M$ that is large compared to $\frac{(m+n)^2}{m}-\frac{m+n}{m}\log \varepsilon$,
\item an $n \times m$-matrix $A$ with entries $0<a_{ij}\leq 1$, where each $a_{ij}=\frac{p_{ij}}{2^M}$ for some integer $p_{ij}$,
\item an integer $q_{\max}<2^M$.
\end{itemize}

When we construct the matrix $B$ in step 1 of the ILLL-algorithm we approximate $c$ as given in~(\ref{eq:lll begin linear system}) by a rational
\begin{equation}
\label{eq: c rational}
\hat{c} = \frac{\lceil 2^M  c\rceil}{2^M} = \frac{ \left \lceil 2^M \left( 2^{-\frac{m+n-1}{4}} \varepsilon  \right)^\frac{m+n}{m}\right \rceil}{2^M}.
\end{equation}
Hence
$
c <  \hat{c} \leq c + \frac{1}{2^M}.
$

In iteration $k$ we use a rational $\hat{c}(k)$ that for $k\geq 2$ is given by
$$
\hat{c}(k) = \frac{\left \lceil 2^M \hat{c}(k-1)2^{-\frac{m+n}{m}}\right \rceil}{2^M} \textrm{ and } \hat{c}(1)=\hat{c} \textrm{ as in~}(\ref{eq: c rational}),
$$
and we change step 4 of the ILLL-algorithm to `multiply the last $m$ rows of $B$ by \mbox{$\hat{c}(k-1)/\hat{c}(k)$}'. The other steps of the rational iterated algorithm are as described in Section~\ref{sec: algorithm}. 

\subsection{The running time of the rational algorithm}
\begin{Theorem}
Let the input  be given as described above. Then the number of arithmetic operations needed by the ILLL-algorithm and the binary length of the integers on which these operations are performed  are both bounded by a polynomial in $m,n$ and $M$.
\end{Theorem}
\begin{proof}
The number of times we apply the LLL-algorithm is not changed by rationalizing $c$, so we find the number of steps $k^\prime$ from Lemma~\ref{lemma: number steps } 
$$
k^\prime = \left\lceil    - \frac{(m+n-1)(m+n)}{4n} + \frac{m \log_2 q_{\max}}{n }   \right\rceil < \left\lceil \frac{mM}{n} \right\rceil .
$$
It is obvious that steps 1, 3, 4 and 5 of the algorithm are polynomial in the size of the input  and we focus on the LLL-step. We determine an upper bound for the length of a basis vector used at the beginning of an iteration in the ILLL-algorithm.

In the first application of the LLL-algorithm the length of the initial basis vectors as given in~(\ref{eq:lll begin linear system}) is bounded by 
$$
|b_i|^2 \leq \displaystyle{\max_j \left\{1, a^2_{1j}+\dots+a^2_{nj}+m\hat{c}^2 \right\}} \leq  m+n, 
  \quad \textrm{ for } 1 \leq i \leq m+n\ .
$$
where we use that $0 <a_{ij} <1$ and $\hat{c} \leq 1$.

The input of each following application of the LLL-algorithm is derived from the reduced basis found in the previous iteration by making some of the entries strictly smaller. Part (ii) of Proposition~\ref{prop: LLL b1} yields that for every vector $b_i$ in a reduced basis it holds that
$$
|b_i|^2 \leq 2^\frac{(m+n)(m+n-1)}{2} (\det(L))^2  \prod_{j=1,j\neq i}^{m+n} |b_i|^{-2}.
$$
The determinant of our starting lattice is given by $\hat{c}^m$ and the determinants of all subsequent lattices are strictly smaller.  Every vector $b_i$ in the lattice is at least as long as the shortest non-zero vector in the lattice. Thus for each~$i$ we have $|b_i|^2 \geq \frac{1}{2^M}$. Combining this yields
$$
|b_i|^2 \leq 2^\frac{(m+n+2M)(m+n-1)}{2} \hat{c}^{2m}   \leq 2^\frac{(m+n+2M)(m+n-1)}{2} 
$$
for every vector used as input for the LLL-step after the first iteration. 

So we have
\begin{equation}
\label{eq: bj bound}
|b_i|^2 < \max \left\{ m+n \,,\,2^\frac{(m+n+2M)(m+n-1)}{2}  \right\} = 2^\frac{(m+n+2M)(m+n-1)}{2}
\end{equation}
for any basis vector that is used as input for an LLL-step in the ILLL-algorithm.

Proposition~\ref{prop: LLL complexity} shows that for a given basis  $b_1,\dots,b_{m+n}$ for $\Z^{m+n}$ with $F \in \R$, $F \geq 2$ such that $|b_i|^2\leq F$ for $1 \leq i \leq m+n$ the number of arithmetic operations needed to find a reduced basis from this input is $O((m+n)^4 \log F)$. For matrices with entries in $\Q$ we need to clear denominators before applying this proposition. Thus for a basis with basis vectors $|b_i|^2\leq F$ and rational entries that can all be written as fractions with denominator $2^M$ the number of arithmetic operations is $O((m+n)^4 \log (2^{2M} F))$.

Combining this with~(\ref{eq: bj bound}) and the number of steps yields the proposition.  
\end{proof}

\subsection{Approximation results from the rational algorithm}
Assume that the input matrix $A$ (with entries $a_{ij}  = \frac{p_{ij}}{2^M} \in \Q$) is an approximation of an $n\times m$-matrix $\mathcal{A}$ (with entries $\alpha_{ij}\in \R$), found by putting $ a_{ij}= \frac{  \lceil 2^M \alpha_{ij} \rceil}{2^M}$. In this subsection we derive the approximation results guaranteed by the rational iterated algorithm for the $\alpha_{ij} \in \R$. 

According to~(\ref{eq: q in c}) and~(\ref{eq: quality in c}) the LLL-algorithm applied with $\hat{c}$ instead of $c$ guarantees to find an $m$-tuple $q_1,\dots,q_m$ such that
\begin{eqnarray*}
q= \max_j |q_j| &\leq&2^\frac{(m+n-1)(m+n)}{4m}\varepsilon^\frac{-n}{m},\\
\textrm{ and } && \\ 
\max_i \| q_1 a_{i1} + \dots + q_m a_{im}\| &\leq& 2^\frac{m+n-1}{4} \left(\left(2^{-\frac{m+n-1}{4}} \varepsilon\right)^\frac{m+n}{m} +\frac{1}{2^M} \right)^\frac{m}{m+n}\\
&\leq& \varepsilon +2^\frac{(m+n-1)(m+n)-4Mm}{4(m+n)},
\end{eqnarray*}
the last inequality follows from the fact that $(x+y)^\alpha \leq x^\alpha + y^\alpha $ if $\alpha<1$ and $x,y>0$.

For the $\alpha_{ij}$ we find that
\[
\begin{aligned}
\max_i \| q_1 \alpha_{i1} + \dots + q_m \alpha_{im}\| &\leq \max_i \| q_1 a_{i1} + \dots + q_m a_{im}\| + mq 2^{-M}\\
& \leq \varepsilon + 2^{\frac{m+n-1}{4}-\frac{Mm}{m+n}} + m \varepsilon^\frac{-n}{m} 2^{\frac{(m+n-1)(m+n)}{4m}-M}.
\end{aligned}
\]
On page~\pageref{sec: polynomial} we have chosen $M$ large enough to guarantee that the error introduced by rationalizing the entries is negligible.

We show that in every step the difference between $\hat{c}(k)$ and $c(k)$ is bounded by $\frac{2}{2^M}$.

\begin{Lemma}
For each integer $k \geq 0$,
$$
c(k) \leq \hat{c}(k) < c(k) + \frac{1}{2^M} \sum_{i=0}^k 2^{-\frac{i(m+n)}{m}} < c(k) + \frac{2}{2^M}.
$$
\end{Lemma}
\begin{proof}
We use induction. For $k=0$ we have $\hat{c}(0) =   \frac{\left \lceil c(0)2^M \right \rceil}{2^M}$ and trivially
$$
c(0) \leq \hat{c}(0) < c(0) + \frac{1}{2^M}.
$$
Assume that $\displaystyle{c(k-1) \leq \hat{c}(k-1) < c(k-1) + \frac{1}{2^M} \sum_{i=0}^{k-1} 2^{-\frac{i(m+n)}{m}}}$ and consider $\hat{c}(k) $. From the definition of $\hat{c}(k)$ and the induction assumption it follows that
$$
\hat{c}(k)  = \frac{\left \lceil \hat{c}(k-1) \,2^{-\frac{m+n}{m}} 2^M \right \rceil}{2^M} \geq \frac{ \hat{c}(k-1) }{2^\frac{m+n}{m}} \geq  \frac{ c(k-1) }{2^\frac{m+n}{m}} =c(k)
$$
and
$$
\begin{aligned}
\hat{c}(k)  = \frac{\left \lceil \hat{c}(k-1) \, 2^{-\frac{m+n}{m}} 2^M \right \rceil}{2^M} &< \frac{\hat{c}(k-1)}{2^\frac{m+n}{m} } + \frac{1}{2^M} \\
&<   \frac{c(k-1) + \frac{1}{2^M} \sum_{i=0}^{k-1} 2^{-\frac{i(m+n)}{m}}}{2^\frac{m+n}{m} } + \frac{1}{2^M} \\
&= c(k) +  \frac{1}{2^M} \sum_{i=0}^{k} 2^{-\frac{i(m+n)}{m}}.
 \end{aligned}
$$
Finally note that $\displaystyle{ \sum_{i=0}^{k} 2^{-\frac{i(m+n)}{m}} <2}$ for all $k$.
\end{proof}

One can derive analogues of Theorem~\ref{th: for each Q}, Lemma~\ref{lem: what we find} and Theorem~\ref{th: what you find} for the polynomial version of the ILLL-algorithm by carefully adjusting for the introduced error. We do not give the details, since in practice this error is negligible.

\section{Experimental data}
\label{sec: numerical}
In this section we present some experimental data from the rational ILLL-algorithm. In our experiments we choose the dimensions $m$ and $n$ and iteration speed $d$. We fill the $m \times n$ matrix $A$ with random numbers in the interval $[0,1]$ and repeat the entire ILLL-algorithm for a large number of these random matrices to find our results. First we look at the distribution of the approximation quality. Then we look at the growth of the denominators $q$ found by the algorithm.

\subsection{The distribution of the approximation qualities}
For one-dimensional continued fractions the approximation coefficients  $\Theta_k$ are defined as 
$$
\Theta_k= q_k^2 \left| a -\frac{p_k}{q_k} \right|,
$$
where $p_k/q_k$ is the $n$th convergent of $a$.

For the multi-dimensional case we define $\Theta_k$ in a similar way
\begin{equation}
\label{eq: theta multi-dimensional}
\Theta_k =  q(k)^\frac{m}{n} \max_i \| q_1(k) \,a_{i1} + \dots + q_m(k) \,a_{im} \|  . 
\end{equation}

\subsubsection{The one-dimensional case $m=n=1$}

In~\cite{BKopt} it was shown that for optimal continued fractions for almost all $a$ one has that

$ \lim_{N \rightarrow \infty} \frac{1}{N} \# \left\{ 1 \leq n \leq N: \Theta_n(x) \leq z \right\} = F(z)$,  where
$$
\begin{aligned}
&F(z) = \begin{cases}
\displaystyle{\frac{z}{\log G}}, &0 \leq z \leq \frac{1}{\sqrt{5}},\\
\\
\displaystyle{\frac{\sqrt{1-4z^2}+\log (G \frac{1-\sqrt{1-4z^2}}{2z}) }{\log G}},  &\frac{1}{\sqrt{5}} \leq z \leq \frac12,\\
\\
1  ,& \frac12 \leq z \leq 1,
\end{cases}
\end{aligned}
$$
where $G = \frac{\sqrt{5}+1}{2}$.

As the name suggests, the optimal continued fraction algorithm gives the optimal approximation results. The denominators it finds, grow with maximal rate and all approximations with $\Theta <\frac12$ are found. 

We plot the distribution of the $\Theta$'s found by the ILLL-algorithm for $m=n=1$ and $d=2$ in Figure~\ref{im: theta m1n1d2}. The ILLL-algorithm might find the same approximation more than once. We see in Figure~\ref{im: theta m1n1d2} that for $d=2$ the distribution function differs depending on whether we leave in the duplicates or sort them out. With the duplicate approximations removed the distribution of $\Theta$ strongly resembles $F(z)$ of the optimal continued fraction. The duplicates that the ILLL-algorithm finds are usually good approximations: if they are much better than necessary they will also be an admissible solution in the next few iterations.

\begin{figure}[!ht]
\includegraphics[height=50mm]{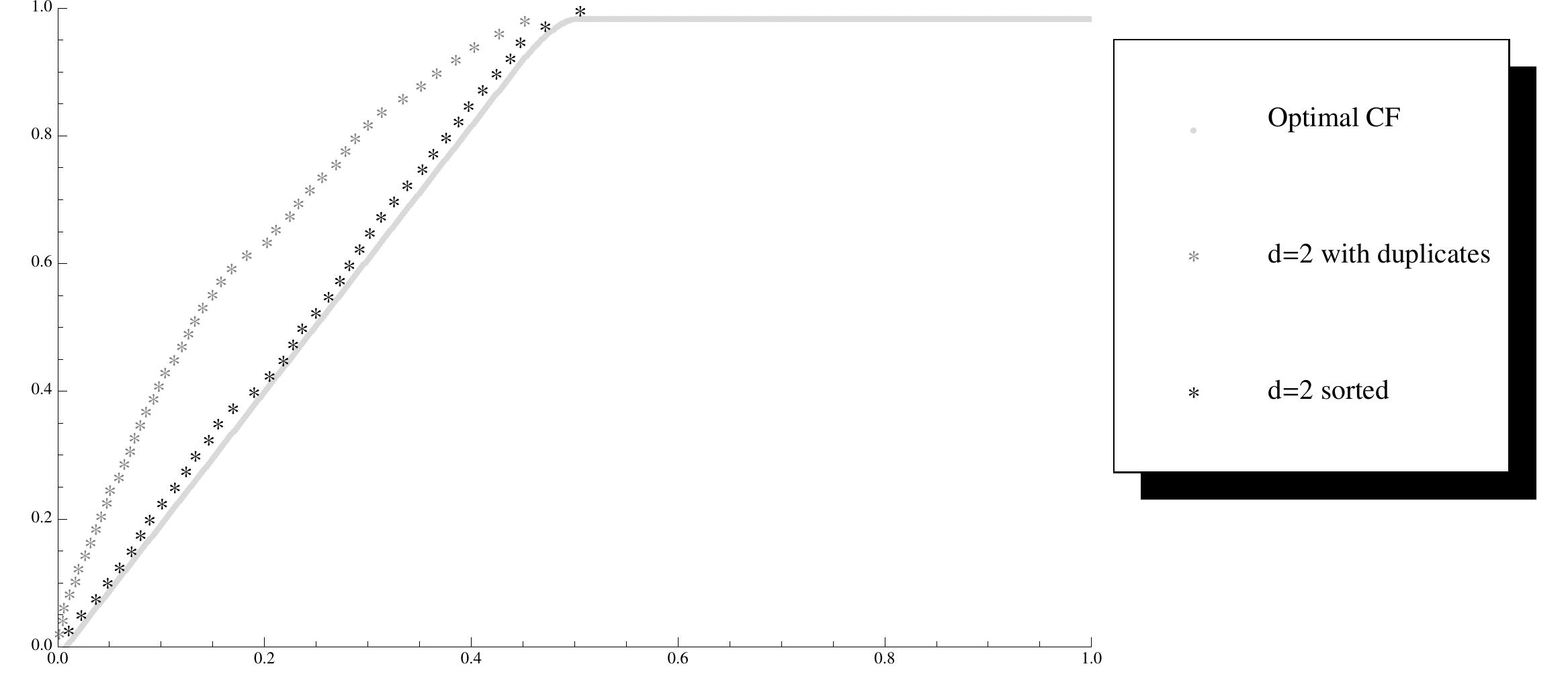}
\caption{The distribution function for $\Theta$ from ILLL with \mbox{$m=n=1$} and  $d=2$, with and without the duplicate approximations, compared to the distribution function of $\Theta$ for optimal continued fractions.}
\label{im: theta m1n1d2}
\end{figure}

For larger $d$ we do not find so many duplicates, because the quality has to improve much more in every step; also see Figure~\ref{im: theta m1n1d64} for an example with $d=64$.

\begin{figure}[!ht]
\includegraphics[height=50mm]{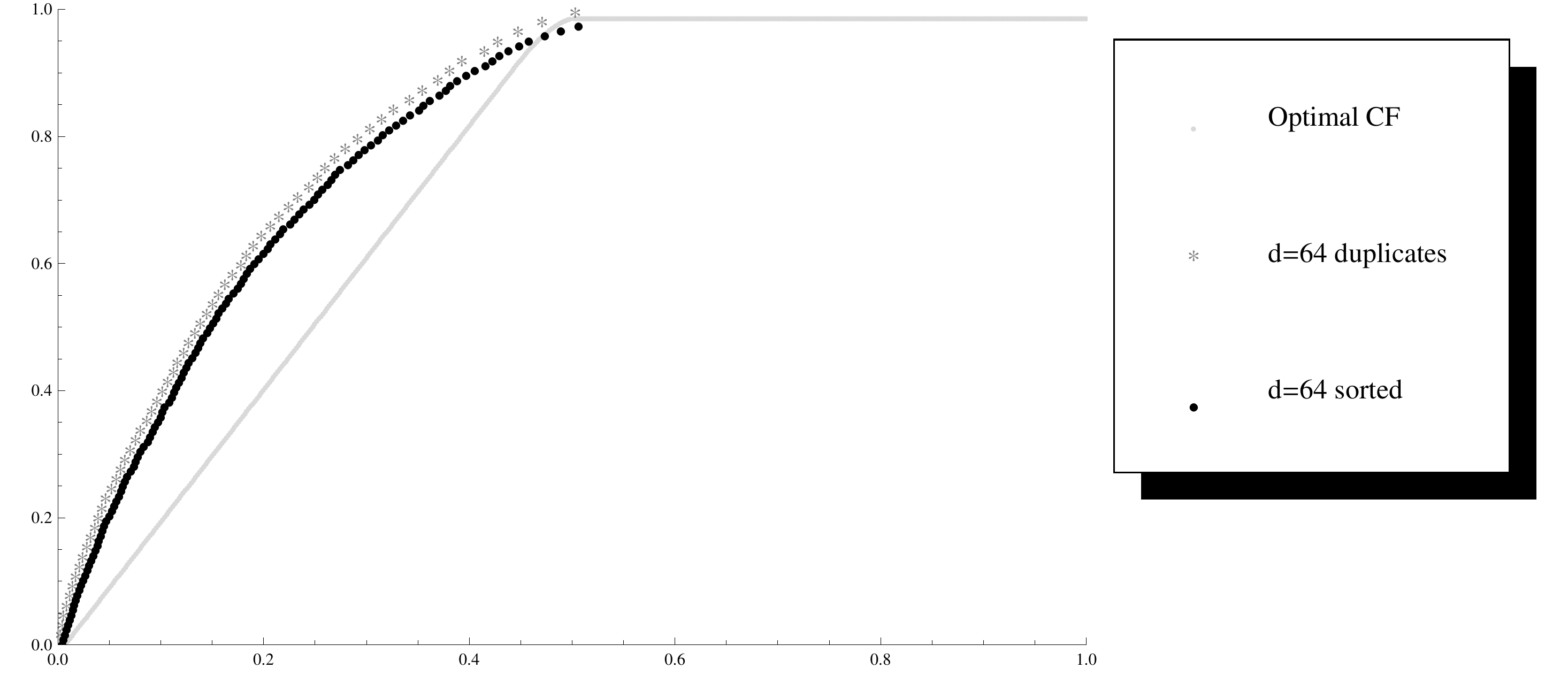}
\caption{The distribution function for $\Theta$ from ILLL with \mbox{$m=n=1$} and  $d=64$, with and without the duplicate approximations, compared to the distribution function of $\Theta$ for optimal continued fractions.}
\label{im: theta m1n1d64}
\end{figure}

From now on we remove duplicates from our results.

\begin{figure}[!ht]
\includegraphics[height=50mm]{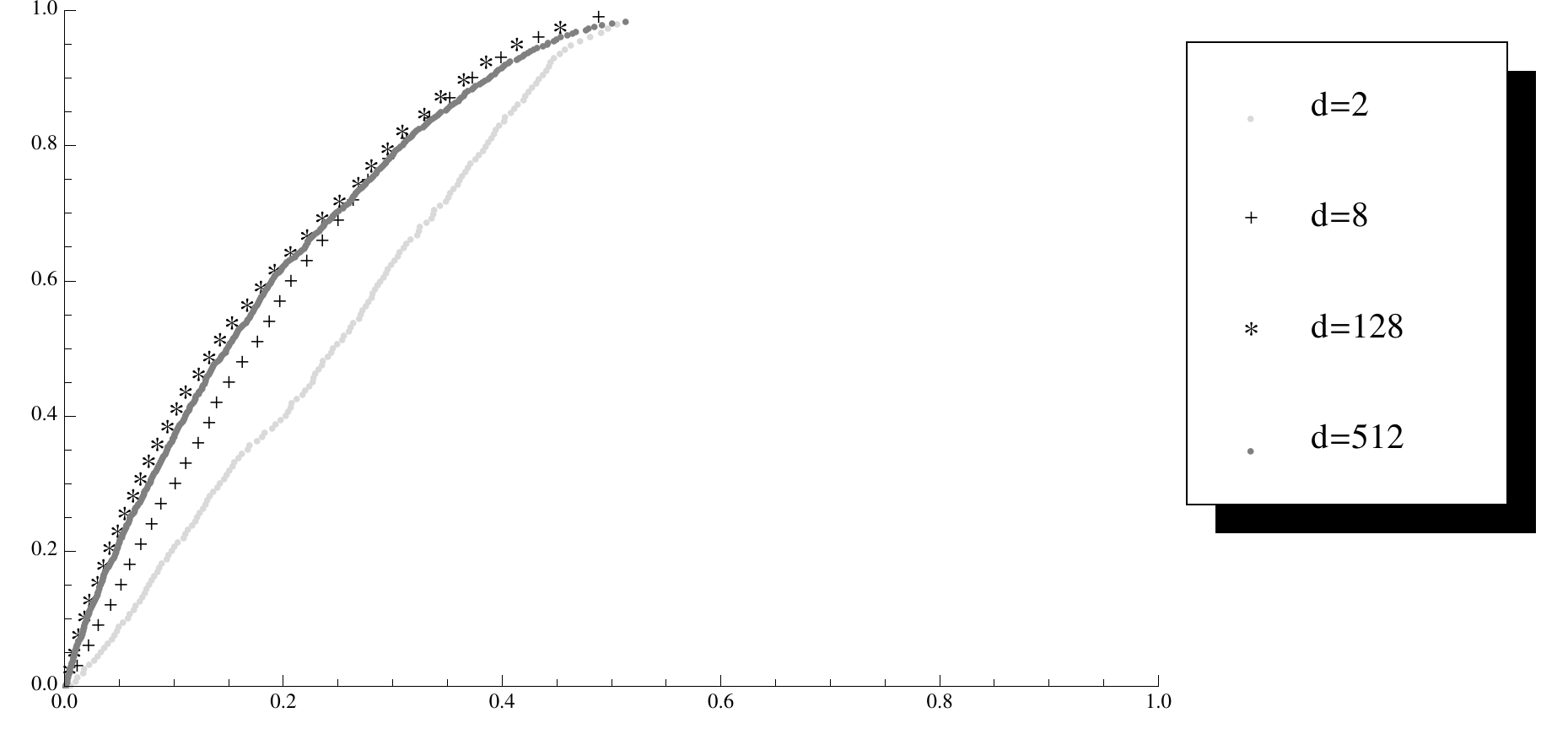}
\caption{The distribution function for $\Theta$ from ILLL (with duplicates removed) with \mbox{$m=n=1$} and various values of $d$.}
\label{im: theta dlimit}
\end{figure}




\subsection{The multi-dimensional case}
In this section we show some results for the distribution of the $\Theta$'s found by the ILLL-algorithm. For fixed $m$ and $n$ there also appears to be a limit distribution for $\Theta$ as $d$ grows. See~Figure~\ref{im: theta m3n2} for an example with $m=3$ and $n=2$, and compare this with Figure~\ref{im: theta dlimit}. In this section we fix $d=512$.

\begin{figure}[!ht]
\includegraphics[height=50mm]{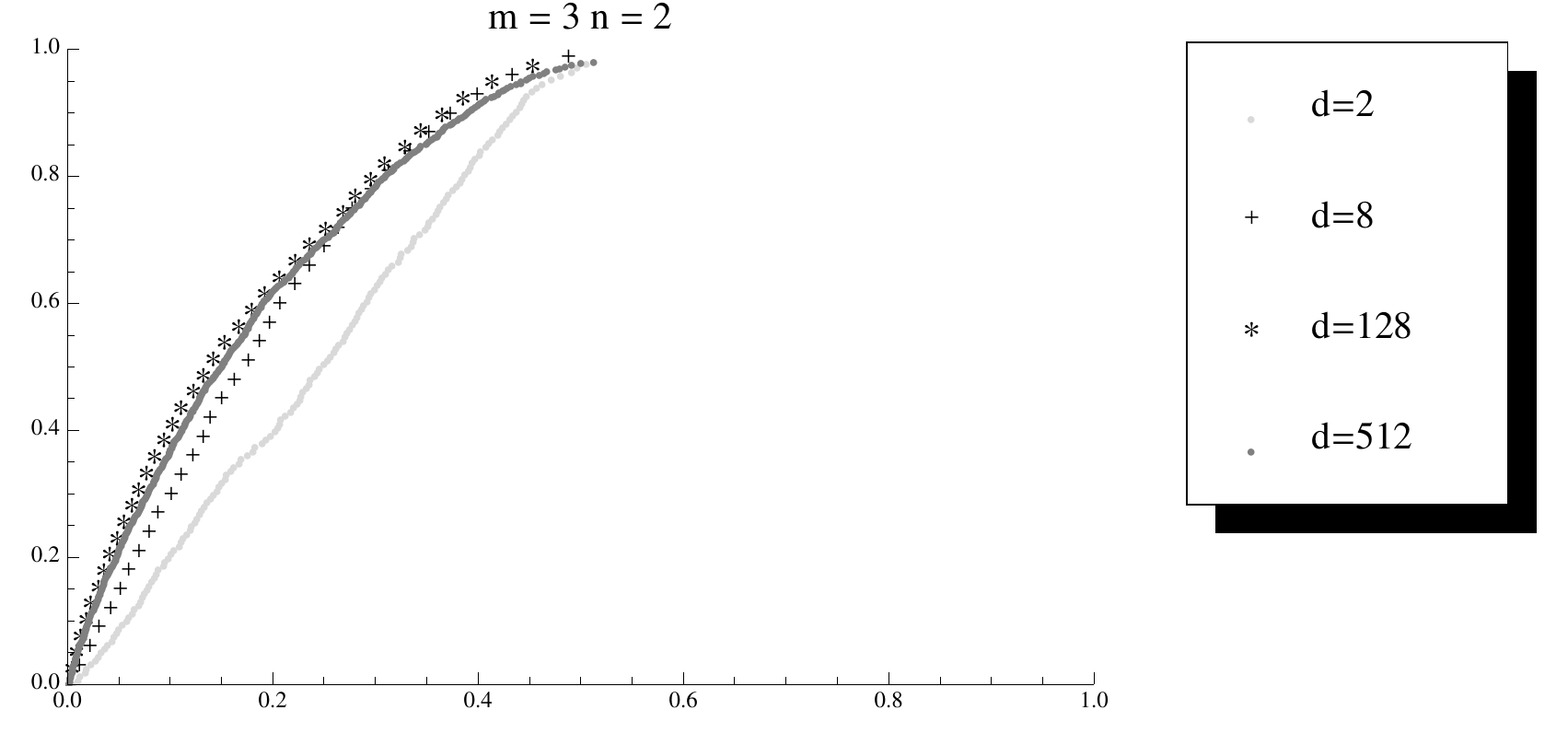}
\caption{The distribution function for $\Theta$ from ILLL with $m=3$ and $n=2$ for $d=2,8,128$ and $512$.}
\label{im: theta m3n2}
\end{figure}

In Figure~\ref{im: theta multi } we show some distributions for cases where either $m$ or $n$ is 1.

\begin{figure}[!ht]
\includegraphics[height=50mm]{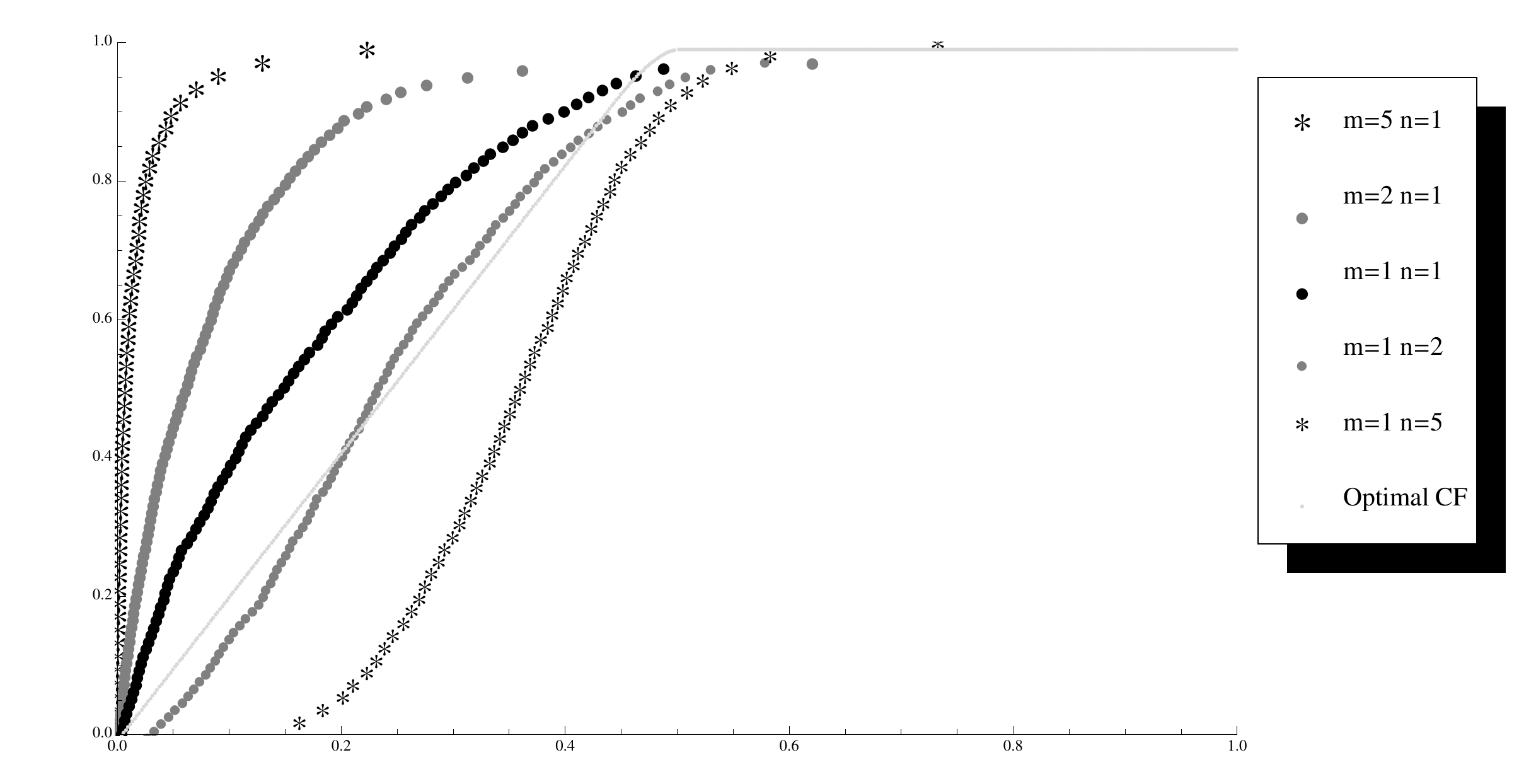}
\caption{The distribution for $\Theta$ from ILLL when either $m=1$ or $n=1$.}
\label{im: theta multi }
\end{figure}

In Figure~\ref{im: theta m=n} we show some distributions for cases where $m=n$.

\begin{figure}[!ht]
\includegraphics[height=50mm]{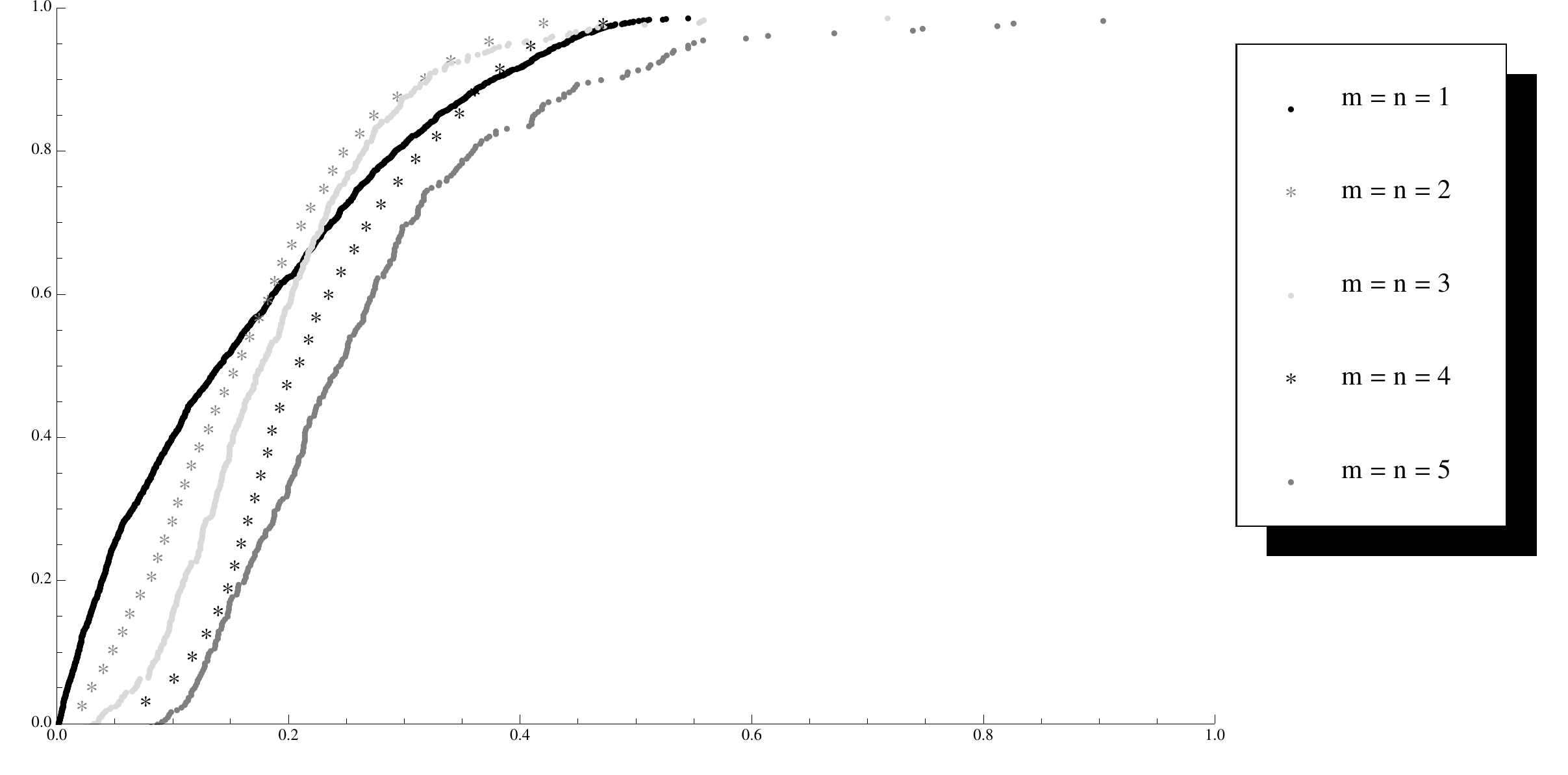}
\caption{The distribution of $\Theta$ from ILLL when $m=n$.}
\label{im: theta m=n}
\end{figure}

\begin{Rmk}
Very rarely the ILLL-algorithm returns an approximation with \mbox{$\Theta>1$}, but this is not visible in the images in this section.
 \end{Rmk}





\subsection{The denominators $q$}
For regular continued fractions, the denominators grow exponentially fast, to be more precise, for almost all $x$ we have that
$$
\lim_{k \rightarrow \infty} q_k^{1/k} = e^\frac{\pi^2}{12\log 2},  
$$ 
see Section~3.5 of~\cite{DK}.

For nearest integer continued fractions the constant $\frac{\pi^2}{12\log 2}$ is replaced by $\frac{\pi^2}{12\log G}$ with $G = \frac{\sqrt{5}+1}{2}$. For multi-dimensional continued fraction algorithms little is known about the distribution of the denominators $q_j$. Lagarias defined in~\cite{Lag1} the notion of a best simultaneous Diophantine approximation and showed that for the ordered denominators  $1=q_1 < q_2 <\dots $ of best approximations for $a_1,\dots,a_n$ it holds that
$$
\lim_{k \rightarrow \infty} \inf q_k^{1/k} \geq 1 + \frac{1}{2^{n+1}}. 
$$
 
\begin{figure}[!ht]
\includegraphics[height=100mm]{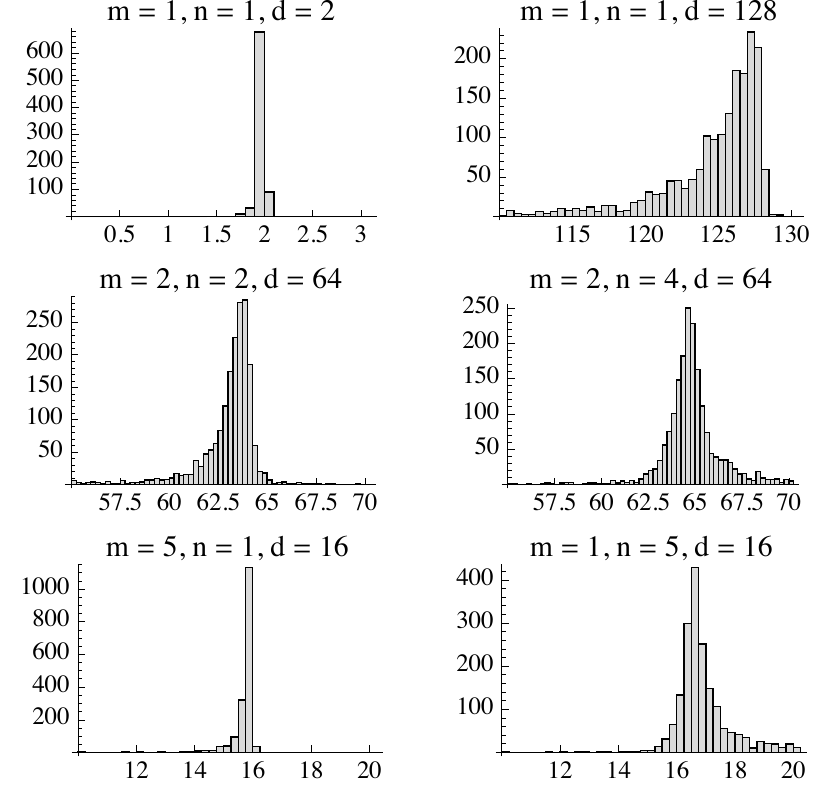}
\caption{Histograms of $e^{\frac{m\log q(k)}{k\,n}}$ for various values of $m,n$ and $d$. In these experiments we used $q_{\max}=10^{40}$ and repeated the ILLL-algorithm $\left\lfloor \frac{2000}{k^\prime}\right\rfloor$ times, with $k^\prime$ from Lemma~\ref{lemma: number steps }.}
\label{im: growth}
\end{figure}

We look at the growth of the denominators $q = \max_j  |q_j|$ that are found by the ILLL-algorithm. Dirichlet's Theorem~\ref{th: ILLL diri} suggests that if $q $ grows exponentially with a rate of $m/n$, then infinitely many approximation with Dirichlet coefficient smaller than 1 can be found. In the iterated LLL-algorithm it is guaranteed by~(\ref{eq: q k bound}) that $q(k)$ is smaller than a constant times $d^\frac{kn}{m}$. Our experiments indicate that $q(k)$ is about $ d^\frac{kn}{m}$, or equivalently that $ e^{\frac{m\log q_k}{k\,n}} $ is about $ d$; see Figure~\ref{im: growth} which gives a histogram of solutions that satisfy $ e^{\frac{m\log q_k}{k\,n}} = x$.

\renewcommand{\bibname}{References} 

\bibliography{bibionica}

\begin{thebibliography}{10}

\bibitem{BKopt}
W.~Bosma and C.~Kraaikamp.
\newblock Metrical theory for optimal continued fractions.
\newblock {\em J. Number Theory}, 34(3):251--270, 1990.

\bibitem{Brentjes}
A.~J. Brentjes.
\newblock {\em Multidimensional continued fraction algorithms}, volume 145 of
  {\em Mathematical Centre Tracts}.
\newblock Mathematisch Centrum, Amsterdam, 1981.

\bibitem{Brun1}
V.~Brun.
\newblock En generalisation av kjedebroken i+ii.
\newblock {\em Skr. Vid. Selsk. Kristiana, Mat. Nat. 6 (1919) and 6 (1920)},
  1919.

\bibitem{DK}
K.~Dajani and C.~Kraaikamp.
\newblock {\em Ergodic theory of numbers}, volume~29 of {\em Carus Mathematical
  Monographs}.
\newblock Mathematical Association of America, Washington, DC, 2002.

\bibitem{FF}
H.~R.~P. Ferguson and R.~W. Forcade.
\newblock Generalization of the {E}uclidean algorithm for real numbers to all
  dimensions higher than two.
\newblock {\em Bull. Amer. Math. Soc. (N.S.)}, 1(6):912--914, 1979.

\bibitem{H}
A.~Hurwitz.
\newblock \"{U}ber die angen\"aherte {D}arstellung der {Z}ahlen durch rationale
  {B}r\"uche.
\newblock {\em Math. Ann.}, 44(2-3):417--436, 1894.

\bibitem{Jacobi}
C.~Jacobi.
\newblock Allgemeine {T}heorie der kettenbruch\"ahnlichen {A}lgorithmen.
\newblock {\em J. Reine Angew. Math.}, 69:29--64, 1868.

\bibitem{Just}
B.~Just.
\newblock Generalizing the continued fraction algorithm to arbitrary
  dimensions.
\newblock {\em SIAM J. Comput.}, 21(5):909--926, 1992.

\bibitem{Lag1}
J.~C. Lagarias.
\newblock Best simultaneously diophantine approximations. i. growth rates of
  best approximation denominators.
\newblock {\em Trans. Amer. Math. Soc.}, 272(2):545--554, 1982.

\bibitem{Lag2}
J.~C. Lagarias.
\newblock The computational complexity of simultaneous {D}iophantine
  approximation problems.
\newblock {\em SIAM J. Comput.}, 14(1):196--209, 1985.

\bibitem{Lag}
J.~C. Lagarias.
\newblock Geodesic multidimensional continued fractions.
\newblock {\em Proc. London Math. Soc. (3)}, 69(3):464--488, 1994.

\bibitem{L}
A.~M. Legendre.
\newblock Essai sur la th\'{e}orie des nombres, 1798.

\bibitem{Di}
G.~Lejeune~Dirichlet.
\newblock {\em Mathematische {W}erke. {B}\"ande {I}, {II}}.
\newblock Herausgegeben auf Veranlassung der K\"oniglich Preussischen Akademie
  der Wissenschaften von L. Kronecker. Chelsea Publishing Co., Bronx, N.Y.,
  1969.

\bibitem{LLL}
A.~K. Lenstra, H.~W. Lenstra, Jr., and L.~Lov{\'a}sz.
\newblock Factoring polynomials with rational coefficients.
\newblock {\em Math. Ann.}, 261(4):515--534, 1982.

\bibitem{Perron}
O.~Perron.
\newblock Grundlagen f\"ur eine {T}heorie des {J}acobischen
  {K}ettenbruchalgorithmus.
\newblock {\em Math. Ann.}, 64(1):1--76, 1907.

\bibitem{Schmidt}
W.~M. Schmidt.
\newblock {\em Diophantine approximation}, volume 785 of {\em Lecture Notes in
  Mathematics}.
\newblock Springer, Berlin, 1980.

\bibitem{Schweiger}
F.~Schweiger.
\newblock {\em Multidimensional continued fractions}.
\newblock Oxford Science Publications. Oxford University Press, Oxford, 2000.

\end{thebibliography}
\bibliographystyle{abbrv}

\end{document}